\documentclass[12pt,reqno]{article}

\usepackage[usenames]{color}
\usepackage{amssymb}
\usepackage{graphicx}
\usepackage{amscd}

\usepackage[colorlinks=true,
linkcolor=webgreen,
filecolor=webbrown,
citecolor=webgreen]{hyperref}

\definecolor{webgreen}{rgb}{0,.5,0}
\definecolor{webbrown}{rgb}{.6,0,0}

\usepackage{color}
\usepackage{fullpage}
\usepackage{float}

\usepackage{graphics,amsmath,amssymb}
\usepackage{amsthm}
\usepackage{amsfonts}
\usepackage{latexsym}
\usepackage{epsf}

\newcommand{\seqnum}[1]{\href{http://www.research.att.com/cgi-bin/access.cgi/as/~njas/sequences/eisA.cgi?Anum=#1}{\underline{#1}}}

\begin{document}


\begin{center}
\vskip 1cm{\LARGE\bf Transforming Recurrent Sequences by Using \\
\vskip .1in
the Binomial and Invert Operators}
\vskip 1cm
\large
Stefano Barbero, Umberto Cerruti, and Nadir Murru\\
Department of Mathematics \\
University of Turin \\
via Carlo Alberto 8/10 \\
Turin \\
Italy \\
\href{mailto:stefano.barbero@unito.it}{\tt stefano.barbero@unito.it}\\
\href{mailto:umberto.cerruti@unito.it}{\tt umberto.cerruti@unito.it}\\
\href{mailto:nadir.murru@unito.it}{\tt nadir.murru@unito.it}\\
\end{center}

\theoremstyle{plain}
\newtheorem{theorem}{Theorem}
\newtheorem{corollary}[theorem]{Corollary}
\newtheorem{lemma}[theorem]{Lemma}
\newtheorem{proposition}[theorem]{Proposition}

\theoremstyle{definition}
\newtheorem{definition}[theorem]{Definition}
\newtheorem{example}[theorem]{Example}
\newtheorem{conjecture}[theorem]{Conjecture}

\theoremstyle{remark}
\newtheorem{rema}[theorem]{Remark}

\begin{abstract}  
In this paper we study the action of the Binomial and Invert
(interpolated) operators on the set of linear recurrent sequences.  We
prove that these operators preserve this set, and we determine how they
change the characteristic polynomials. We show that these operators,
with the aid of two other elementary operators (essentially the left
and right shifts), can transform any impulse sequence (a linear
recurrent sequence starting from $(0,\ldots,0,1)$) into any other
impulse sequence, by two processes that we call \emph{construction} and
\emph{deconstruction}. Finally, we give some applications to polynomial
sequences and pyramidal numbers. We also find a new identity on
Fibonacci numbers, and we prove that $r$--bonacci numbers are a Bell
polynomial transform of the $(r-1)$--bonacci numbers.
\end{abstract}

\section{Introduction}
A rich research field arises from investigation on the operators acting over sequences. In particular, two of these operators, Binomial and Invert, have been deeply studied. So we want to point out some new aspects, giving an in--depth examination of their action  over the set of linear recurrent sequences.      
\begin{definition}
We define, over an integral domain $R$, the set $\mathcal{S}(R)$ of sequences $a=(a_n)_{n=0}^{+\infty}$,  and the set $\mathcal{W}(R)$ of linear recurrent sequences.
\end{definition}

\begin{definition} \label{I-L}
We recall the definition of the well-known operators $I$ and $L$, called \emph{Invert} and \emph{Binomial} respectively.
They act on elements of $\mathcal{S}(R)$ as follows:
\begin{equation*} 
I(a)=b, \ \ \ \sum_{n=0}^{\infty}b_nt^n=\cfrac{\sum_{n=0}^{\infty}a_nt^n}{1-t\sum_{n=0}^{\infty}a_nt^n}
\end{equation*}
\begin{equation*} 
L(a)=c, \ \ \ c_n=\sum_{i=0}^{n}\binom{n}{i}a_i\quad.
\end{equation*}
\end{definition}
The operators $I$ and $L$ can be \emph{iterated} \cite{Spivey} and \emph{interpolated} \cite{Bacher}, becoming $I^{(x)}$ and $L^{(y)}$,  as we show in the next two definitions.

\begin{definition}\label{invert}
The \emph{Invert interpolated operator} $I^{(x)}$, with parameter $x \in R$, transforms any sequence $a\in\mathcal{S}(R)$, having ordinary generating function $A(t)$, into a sequence $b=I^{(x)}(a) \in \mathcal{S}(R)$ having ordinary generating function
\begin{equation}\label{invertgen}
B(t)=\cfrac{A(t)}{1-xtA(t)}\quad.  
\end{equation}
\end{definition}
\begin{definition} \label{binomial}
The \emph{Binomial interpolated operator}$L^{(y)}$, with parameter $y \in R$, transforms any sequence $a\in\mathcal{S}(R)$ into a sequence $c=L^{(y)}(a)\in \mathcal{S}(R)$, whose terms are 
\begin{equation*}
c_n=\sum_{i=0}^{n}\binom{n}{i}y^{n-i}a_i \quad.
\end{equation*}
\end{definition}
For our purposes we also introduce the following two operators
\begin{definition} \label{sigma-rho}
The \emph{right-shift operator} $\sigma$ and the \emph{left-shift operator} $\rho$ change any sequence $a\in\mathcal{S}(R)$ as follows:
\begin{equation*} 
\sigma(a)=(a_1,a_2,a_3,\ldots)\quad \rho(a)=(0,a_0,a_1,a_2,\ldots)\quad.
\end{equation*}
\end{definition}
\begin{rema}\label{invertrec}
The sequence $b=I^{(x)}(a)$ is characterized by the following recurrence
\begin{equation}\label{invertreceq}
\left\{ \begin{array}{l}
b_n  = a_n  + x\sum\limits_{j = 0}^{n - 1} a_{n - 1 - j} b_j; \\ 
b_0  = a_0.\\ 
\end{array} \right. 
\end{equation}
This relation is a straightforward consequence of Definition \ref{invert}. In fact, if we consider the sequence
$u=(1,0,0,\ldots)$ and the convolutional product $*$ between sequences, (\ref{invertgen}) becomes 
$$b = a* (u - x\rho(a))^{ - 1},$$
where $(u-x\rho(a))^{-1}$ is the convolutional inverse of $u-x\rho(a)$.
So 
$$b=a+x\rho (a)*b,$$
 from which (\ref{invertreceq}) follows easily.
\end{rema}

In a previous paper \cite{SteUmb}, Barbero and Cerruti proved that the operators $L^{(y)}, I^{(x)}$ generate a group which is commutative. We want to generalize these results, analyzing the action of these operators on the set $\mathcal{W}(R)$ of all linear recurrent sequences over $R$. We will prove that the group generated  by $L^{(y)}$ and $ I^{(x)}$ acts on $\mathcal{W}(R)$ preserving it, as we will show in the next section.
\section{A commutative group action on the set of linear recurrent sequences} 
First of all, we focus our attention on the operator $L^{(y)}$, with the aim to know how it acts on the characteristic polynomial of a given sequence $a \in\mathcal{W}(R)$. In order to do so, we prove the 
\begin{lemma} \label{lempoly}
Let us consider a field $\mathbb F$ containing $R$ in which we choose an element $\alpha$ and a nonzero element $y$. For any polynomial $P(t)=\sum\limits_{j = 0}^r {x_j t^j }$ over $\mathbb F$ we have
$$ \sum_{i=0}^m\binom{m}{i}y^{i}\alpha^{m-i}P(i)=(\alpha+y)^mQ(m), $$
where $Q(m)=\sum\limits_{j = 0}^r {y_j m^j }$ and $y_j=y_j(x_0,\cdots,x_r,y,\alpha)\in \mathbb F$.
\end{lemma}

\begin{proof}
We use the classical Newton formula
$$ \sum_{i=0}^m\binom{m}{i}y^i\alpha^{m-i}=(\alpha+y)^m.$$
Differentiating with respect to $y$, one gets
$$ \cfrac{1}{y}\sum_{i=0}^m\binom{m}{i}y^i\alpha^{m-i}i=m(\alpha+y)^{m-1}, $$
so
$$ \sum_{i=0}^m\binom{m}{i}y^i\alpha^{m-i}i=ym(\alpha+y)^{m-1}=\cfrac{y}{\alpha+y}m(\alpha+y)^m\quad. $$
We can prove by induction that
\begin{equation} \label{pseudo-binomial} \sum_{i=0}^m\binom{m}{i}y^i\alpha^{m-i}i^s=c_s(m,\alpha,y)(\alpha+y)^m\quad, \end{equation}
for every integer $s\geq1$, where $c_s(m,\alpha,y)$ is a coefficient depending on $m,\alpha,y$ and it can be viewed as a polynomial of degree $s$ in the variable $m$. The basis of the induction has already been proved. So let us suppose that the above formula is true for every integer between $1$ and $s$, and we prove it for $s+1$. Differentiating (\ref{pseudo-binomial}) with respect to $y$, one gets
$$ \cfrac{1}{y}\sum_{i=0}^m\binom{m}{i}y^i\alpha^{m-i}i^{s+1}=c_s'(m,\alpha,y)(\alpha+y)^m+c_s(m,\alpha,y)m(\alpha+y)^{m-1}, $$

$$  \sum_{i=0}^m\binom{m}{i}y^i\alpha^{m-i}i^{s+1}=\left(yc_s'(m,\alpha,y)+\cfrac{yc_s(m,\alpha,y)m}{\alpha+y}\right)(\alpha+y)^m=c_{s+1}(m,\alpha,y)(\alpha+y)^m , $$
where $c_s'(m,\alpha,y)$ is the derivative of $c_s(m,\alpha,y)$ with respect to $y$ and $c_{s+1}(m,\alpha,y)$ is a polynomial of degree $s+1$ in $m$, by the induction hypothesis. Therefore we can complete the proof of the lemma
$$ \sum_{i=0}^m\binom{m}{i}y^{i}P(i)\alpha^{m-i}=
\sum\limits_{j = 0}^r {x_j \sum\limits_{i = 0}^m{ \binom{m}{i} y^i \alpha ^{m - i} i^j }} = (\alpha  + y)^m \sum\limits_{j = 0}^r {x_j c_j (m,\alpha ,y)}=Q(m)(\alpha+y)^m. $$

\end{proof}
In the next proposition we show how to write $c_s(m,\alpha,y)$ of (\ref{pseudo-binomial}), for every natural number $s$ and $m$, in order to obtain $Q(m)$.

\begin{proposition}\label{propcy}
The polynomial $c_s(m,\alpha,y)$  has the following expression 
$$ 
c_s (m,\alpha ,y) = \sum\limits_{h = 0}^s {\left( {\sum\limits_{k = h}^s {\left\{ {\begin{array}{*{20}c}
   s  \\
   k  \\
\end{array}} \right\}\left[ {\begin{array}{*{20}c}
   k  \\
   h  \\
\end{array}} \right]( - 1)^{k - h} \left( {\frac{y}{{\alpha  + y}}} \right)^k } } \right)} m^h \quad.
$$
Here  ${\left\{ {\begin{array}{*{20}c}
   s  \\
   k  \\
\end{array}} \right\}}$ and  ${\left[ {\begin{array}{*{20}c}
   k  \\
   h  \\
\end{array}} \right]}$ are the Stirling numbers of second and first kind respectively, as defined in Graham, Knuth and Patashnik \cite{Concrete}.

\end{proposition}

\begin{proof}
We use the finite differences of a function $f$, defined as follows:
$$ \Delta^nf(t)=\sum_{k=0}^n\binom{n}{k}(-1)^{n-k}f(t+k), $$
which enable us to express $f(t)$ in this way \cite{Concrete}
$$ f(t)=\Delta^df(0)\binom{t}{d}+\Delta^{d-1}f(0)\binom{t}{d-1}+\cdots+\Delta f(0)\binom{t}{1}+f(0)\binom{t}{0}\quad. $$
In the special case $f(i)=i^s$, we obtain
\small
\begin{equation} \label{deltaf0} i^s=\sum_{k=0}^s\Delta^kf(0)\binom{i}{k}=\sum_{k=0}^s\left(\sum_{j=0}^k\binom{k}{j}(-1)^{k-j}j^s\right)\binom{i}{k}=\sum_{k=0}^s
k!\left\{ {\begin{array}{*{20}c}
   s  \\
   k  \\
\end{array}} \right\}
\binom{i}{k}, \end{equation}
\normalsize
where we use the relation \cite{Concrete}
$$
\sum\limits_{j = 0}^k {\left( {\begin{array}{*{20}c}
   k  \\
   j  \\
\end{array}} \right)} ( - 1)^{k - j} j^s  = k!\left\{ {\begin{array}{*{20}c}
   s  \\
   k  \\
\end{array}} \right\}.
$$
Now it is clear that
\begin{equation} \label{pseudo-binomial2} \sum_{i=0}^m\binom{m}{i}\alpha^{m-i}y^ii^s=\sum_{k=0}^s\Delta^kf(0)\sum_{i=0}^m\binom{m}{i}\binom{i}{k}\alpha^{m-i}y^i. \end{equation}
Considering the second sum at the right side, we have
$$ \sum_{i=0}^m\binom{m}{i}\binom{i}{k}\alpha^{m-i}y^i=\binom{m}{k}\sum_{i=k}^m\binom{m-k}{i-k}\alpha^{m-i}y^i=\binom{m}{k}y^k(\alpha+y)^{m-k}.$$
If we put this relation into (\ref{pseudo-binomial2}), using (\ref{deltaf0}), we obtain 
$$ \sum_{i=0}^m\binom{m}{i}\alpha^{m-i}y^ii^s=
\sum\limits_{k = 0}^s {k!\left\{ {\begin{array}{*{20}c}
   s  \\
   k  \\
\end{array}} \right\}\binom{m}{k}} \left( {\frac{y}{{\alpha  + y}}} \right)^k (\alpha  + y)^m  = \sum\limits_{k = 0}^s {\left\{ {\begin{array}{*{20}c}
   s  \\
   k  \\
\end{array}} \right\}} \left( {\frac{y}{{\alpha  + y}}} \right)^k m^{\underline k } (\alpha  + y)^m,$$
where $m^{\underline k }$ denotes the falling factorial $k$th-power of $m$ \cite{Concrete}. We can write
\begin{equation}\label{falling}
m^{\underline k }  = \sum\limits_{h = 0}^k {\left[ {\begin{array}{*{20}c}
   k  \\
   h  \\
\end{array}} \right]} ( - 1)^{k - h} m^h .
\end{equation}
Therefore the proof is complete, since we can use (\ref{falling}) in the previous equality.
\end{proof}

\begin{corollary}\label{coro}

With the hypotheses of Lemma \ref{lempoly} and Proposition \ref{propcy}

$$ Q(m)=\sum_{i=0}^r \sum\limits_{h = 0}^i {{\sum\limits_{k = h}^i {x_i\left\{ {\begin{array}{*{20}c}
   i  \\
   k  \\
\end{array}} \right\}\left[ {\begin{array}{*{20}c}
   k  \\
   h  \\
\end{array}} \right]( - 1)^{k - h} \left( {\frac{y}{{\alpha  + y}}} \right)^k } } } m^h. 
$$

\end{corollary}

Now we are ready to point out the effects of $L^{(y)}$ on characteristic polynomials, as we claim in the following

\begin{theorem} \label{binomial-theorem}
Let $a\in \mathcal{W}(R)$ be a linear recurrent sequence of degree $r$, with characteristic polynomial over $R$ given by $f(t)=t^r-\sum_{i=1}^r h_it^{r-i}$ and zeros $\alpha_1,\ldots,\alpha_r$, over a field $\mathbb{F}$ containing $R$. Then $c=L^{(y)}(a)$ is a linear recurrent sequence of degree $r$, with characteristic polynomial $L^{(y)}(f(t))=f(t-y)$ and zeros $y+\alpha_1,\ldots,y+\alpha_r$, over the same field $\mathbb{F}$. 
\end{theorem}

\begin{proof}

If $f(t)$ has distinct zeros $\alpha_1,\ldots,\alpha_r$, by Binet formula
$$ a_n=\sum_{i=0}^{r}A_i\alpha_{i}^{n} , $$
for some $A_i \in \mathbb F$ derived from initial conditions. For all terms of $c$, we have 
$$ c_n=\sum_{i=0}^{n}\binom{n}{i}y^{n-i}a_i=\sum_{i=0}^{n}\binom{n}{i}y^{n-i}\sum_{j=0}^{r}A_j\alpha_{j}^{i}=\sum_{j=0}^{r}A_j(\alpha_j+y)^n\quad.$$
If $f(t)$ has a multiple zero $\alpha_r$ of multiplicity $s$, the Binet formula becomes
$$  a_n=\sum_{i=0}^{r-s}A_i\alpha_i^n+A_r(n)\alpha_r^n , $$
where $A_i \in \mathbb F$, for $i=1,\ldots,r-s$, and $A_r(n)$ is a polynomial of degree $s-1$ over $\mathbb F$.
In this case
$$ c_n=\sum\limits_{j = 1}^{r - s} {A_j } (\alpha _j  + y)^n +\sum_{i=0}^{n}\binom{n}{i}y^{i}A_r(n-i)\alpha_r^{n-i}, $$ 
and by Lemma \ref{lempoly}
$$ c_n=\sum\limits_{j = 1}^{r - s} {A_j } (\alpha _j  + y)^n +Q(n)(\alpha_r+y)^n, $$
where $Q(n)$ is a polynomial of degree $s-1$ in $n$, as shown in Proposition \ref{propcy} and Corollary \ref{coro}. In order to complete the proof, we have to observe that this method can be applied, even though there would be more than one multiple zero. 

\end{proof}

In the next corollary we write explicitly the characteristic polynomial $L^{(y)}(f(t))$ .

\begin{corollary} \label{Lf-theorem}

Let $f(t)=\sum_{k=0}^{r} h_k t^{r-k}$ be a polynomial over R, with $h_0=1$. Then
$$ L^{(y)}(f(t))=\sum_{k=0}^{r}p_k(y,h_1,\ldots,h_k)t^{r-k}, $$
where
$$ p_k(y,h_1,\ldots,h_k)=\sum_{i=0}^k\binom{r-i}{k-i}h_i(-y)^{k-i}. $$
\end{corollary}

\begin{proof}
Using the results of Theorem \ref{binomial-theorem}, we have
$$ L^{(y)}(f(t))=f(t-y)=\sum_{i=0}^{r}h_i(t-y)^{r-i}.$$
Developing $(t-y)^{r-i}$ and rearranging sums, we obtain
$$L^{(y)}(f(t))=\sum_{i=0}^{r}h_i\sum_{k=i}^{r}\binom{r-i}{k-i}(-y)^{k-i}t^{r-k}=\sum_{k=0}^{r}\left(\sum_{i=0}^k\binom{r-i}{k-i}h_i(-y)^{k-i}\right)t^{r-k}\quad.$$

\end{proof}
\begin{rema}\label{fignum} 
The polynomials $p_k(y,h_1,\ldots,h_k)$ have coefficients related to triangular numbers, tetrahedral numbers and their generalizations. In order to make clear this relation, we introduce the sequences
$(T^{(k)}_h)_{h=0}^{+\infty}$ of \emph{figurate numbers}. Starting from $(T^{(1)}_h)_{h=0}^{+\infty}=\{0,1,1,1,1,1,\ldots\}$, each term of  $(T^{(k)}_h)_{h=0}^{+\infty}$ corresponds to a  partial sum of $(T_h^{(k-1)})_{h=0}^{+\infty}$, for $k=1,2,3,...$ . 

For example,
$(T^{(2)}_h)_{h=0}^{+\infty}=\{0,1,2,3,4,5,\ldots\}$, $(T^{(3)}_h)_{h=0}^{+\infty}=\{0,1,3,6,10,15,\ldots\}$ (triangular numbers \seqnum{A000217}). The $h$th term of $(T^{(k)}_h)$ is \cite{Dickson}
$$ T^{(k)}_h=\binom{h+k-2}{k-1}\quad. $$

In particular,
$$ T^{(r-k+1)}_{k-i+1}=\binom{r-k+1+k-i+1-2}{k-i}=\binom{r-i}{k-i},$$
therefore the polynomials $p_k(y,h_1,\ldots,h_k)$ are connected to figurate numbers
$$ p_k(y,h_1,\ldots,h_k)=\sum_{i=0}^k T^{(r-k+1)}_{k-i+1}h_i(-y)^{k-i}\quad.$$
\end{rema}
Now we study the action of $I^{(x)}$ on $\mathcal{W}(R)$. We recall that any linear recurrent sequence $a \in \mathcal{W}(R)$, of degree $r$, with characteristic polynomial $f(t)=t^r-\sum_{i=1}^{r}h_it^{r-i}$, has a rational generating function \cite{Concrete} $$A(t)=\cfrac{u(t)}{f^R(t)}\quad.$$
Here  $$f^R(t)=1-\sum_{i=1}^{r}h_{i}t^{i}$$  denotes the reflected polynomial of $f(t)$. Moreover $u(t)=\sum_{i=0}^{r-1}u_{i}t^{i}$ is a polynomial over $R$, 
wwhose coefficients $u_{i}$ can be derived from initial conditions $$(a_0,a_1,\ldots,a_{r-1})=(s_0,s_1,\ldots,s_{r-1}).$$

\begin{theorem} \label{invert-theorem}
Let $a\in \mathcal{W}(R)$ be a linear recurrent sequence of degree $r$, with characteristic polynomial over $R$ given by $f(t)=t^r-\sum_{i=1}^{r}h_it^{r-i}$, and generating function $A(t)=\cfrac{u(t)}{f^R(t)}\quad.$ Then $b=I^{(x)}(a)$ is a linear recurrent sequence, with characteristic polynomial $$(f^R(t)-xtu(t))^R,$$ whose coefficients are
\begin{equation}\label{Ifc}
h_1+xs_0, \quad h_{i+1}  + xs_i  - x\sum\limits_{j = 1}^{i} {h_j } s_{i - j} \quad \text{for} \ i=1,\ldots,r-1\quad.
\end{equation}

\end{theorem}

\begin{proof}

By Definition \ref{invert}, $b=I^{(x)}(a)$ has the rational generating function $$B(t)=\cfrac{A(t)}{1-xtA(t)}\quad.$$ Substituting the rational generating function of $a$ we obtain

$$ B(t)=\cfrac{A(t)}{1-xtA(t)}=\cfrac{\cfrac{u(t)}{f^R(t)}}{1-xt\cfrac{u(t)}{f^R(t)}}=\cfrac{u(t)}{f^R(t)-xtu(t)}=\cfrac{u(t)}{[(f^R(t)-xtu(t))^R]^R} \quad. $$
This relation immediately proves that $I^{(x)}(a)$ is a linear recurrent sequence having characteristic polynomial $(f^R(t)-xtu(t))^R$. If we explicitly write this polynomial, we find
\begin{equation} \label{IfR} (f^R(t)-xtu(t))^R=(-(h_r+xu_{r-1})t^r-(h_{r-1}+xu_{r-2})t^{r-1}-\cdots-(h_1+xu_0)t+1)^R, \end{equation}
and clearly
\begin{equation}\label{Ifp} 
(f^R(t)-xtu(t))^R=t^r-(h_1+xu_0)t^{r-1}-\cdots-(h_{r-1}+xu_{r-2})t-(h_r+xu_{r-1}). 
\end{equation}
Now, to complete the proof, we only need to observe that 
$$u_0=s_0,\quad u_i=s_i-\sum_{j=1}^{i}h_js_{i-j} \quad i=1,\ldots,r-1 .$$
In fact the terms $u_i$ can be easily found, by means of a simple comparison between coefficients of $t$ on the left and right side of the equality $$A(t)f^R(t)=u(t), $$ and using initial conditions $(a_0,a_1,\ldots,a_{r-1})=(s_0,s_1,\ldots,s_{r-1}).$
\end{proof}
\begin{corollary}
Any two polynomials $f(t)=t^r-\sum_{i=1}^{r}h_it^{r-i}$ and $g(t)=t^r-\sum_{i=1}^{r}H_it^{r-i}$ over $R$ can be viewed as  characteristic polynomials of two uniquely determined sequences $a\in\mathcal{W}(R)$ and $b=I^{(x)}(a)\in\mathcal{W}(R)$, respectively, where $x$ is an invertible element of $R$.
\end{corollary}
\begin{proof}
As a consequence of (\ref{Ifc}) we can easily solve the system
$$h_1+xs_0=H_1,\quad  h_{i+1}  + xs_i  - x\sum\limits_{j = 1}^{i} {h_j } s_{i - j}=H_{i+1} \quad i=1,\ldots,r-1\quad,$$
in order to determine the unique values of initial conditions on $a$ such that $b=I^{(x)}(a)$.
\end{proof}

\begin{rema}\label{annihil}
The previous theorem shows how the action of $I^{(x)}$ on $\mathcal{W}(R)$ is related to the initial conditions on $a$.
In particular, for any $a\in\mathcal{W}(R)$, the characteristic polynomial of $b=I^{(x)}(a)$ can have degree less than the one of $a$. In other words, under a suitable choice of $x$, we can reduce the degree of polynomial (\ref{IfR}) changing his explicit form (\ref{Ifp}). In fact, when $u_{r-1}$ is an invertible element of $R$, we can choose $x=-h_r(u_{r-1})^{-1}$ in order to annihilate the coefficient of $t^r$ in (\ref{IfR}).
\end{rema}

\begin{rema}\label{comb}
The Invert operator also has a combinatorial interpretation. If $b=I(a)$, then $b_n$ is the number of ordered arrangements of postage stamps, of total value $n+1$, that can be formed if we have $a_i$ types of stamps of value $i+1$, when $i\geq 0$ \cite{BerSlo}. We want to point out a connection between Invert and complete ordinary Bell polynomials, as they are presented in Port \cite{Port}, arising from their combinatorial interpretations. So we recall that complete ordinary Bell polynomials are defined  by (Port \cite{Port})
$$ B_n(t)=\sum_{k=1}^nB_{n,k}(t), $$
where $t=(t_1,t_2,\ldots)$ and the partial ordinary Bell polynomials $B_{n,k}(t)$ satisfy
$$ \left(\sum_{n\geq1}t_nz^n\right)^k=\sum_{n \geq k}B_{n,k}(t)z^n, $$
and
$$ 
B_{n,k} \left( t \right) = \sum_{\substack {i_1  + 2i_2  +  \cdots  + ni_n = n \\ i_1  + i_2  +  \cdots  + i_n  = k}} \frac{{k!}}{{i_1 !i_2 ! \cdots i_n !}}t_1^{i_1 } t_2^{i_2 }  \cdots t_n^{i_n }\quad.
$$
The last relation provides a straightforward combinatorial interpretation for complete ordinary Bell polynomials.
In fact $B_n(t)$ corresponds to the sum of monomials $t_{i_1}\cdots t_{i_k}$ (written with repetition of $t_{i_j}$), whose coefficients are the number of way to write the word $t_{i_1}\cdots t_{i_k}$, when $\sum_{j=1}^{k}i_j=n$. 
For example,
$$ B_4(t)=t_4+2t_1t_3+t_2t_2+3t_1t_1t_2+t_1t_1t_1t_1=t_4+2t_1t_3+t_2^2+3t_1^2t_2+t_1^4. $$
Now, if we consider the sequence $b=I(a)$, $ a\in\mathcal{S}(R)$, we can put together the combinatorial aspects of Invert and complete ordinary Bell polynomials
$$ b_n=B_{n+1}(a).$$
\end{rema}
\section{Construction and Deconstruction of impulsequences}
In this section we explain how previous results can create an interesting set of tools, which will enable us to play with sequences $a\in\mathcal{W}(R)$, of order $r$, having the \emph{r--impulse} $(0,0,\ldots,0,1)$ as initial conditions. Let us call such sequences \emph{impulsequences}. 
Considering the impulsequence $a$, with characteristic polynomial $f(t)$ of degree $r$, from Theorem \ref{binomial-theorem} and Theorem \ref{invert-theorem}, respectively, we have
\begin{equation} \label{Lf} L^{(z)}(f(t))=f(t-z),\end{equation}
and
\begin{equation} \label{If} I^{(z)}(f(t))=f(t)-z. \end{equation}
Using these relations we can \emph{construct} every impulsequence, from the \emph{startsequence} $$u=(1,0,0,0,\ldots).$$ Viceversa it is possible to \emph{deconstruct} every impulse sequence into  $u$.

The sense of these claims will soon be clear. We also need the two operators $\sigma$ and $\rho$, introduced in Definition \ref{sigma-rho}. They act on a linear recurrent sequence changing its characteristic polynomial, dividing or multiplying it by $t$, respectively, when possible.
So let us start a \emph{L--construction} from $u$.
The startsequence $u=r^{(0)}$ has characteristic polynomial $f(t)=t$, with zero $\{0\}$. The operator $L^{(z_1)}$, applied to $r^{(0)}$, translates by $z_1$ the zero of $f(t)$. The characteristic polynomial becomes $f(t-z_1)=t-z_1$ and so
$$ v^{(1)}=L^{(z_1)}(r^{(0)})=(1,z_1,z_1^2,z_1^3,\ldots). $$
Using the operator $\rho$, we obtain
$$ \rho(v^{(1)})=(0,1,z_1,z_1^2,z_1^3,\ldots)=r^{(1)}. $$
The sequence $r^{(1)}$ has characteristic polynomial $t(t-z_1)=t^2-z_1t$ with zeros $\{0,z_1\}$. Acting on $r^{(1)}$ with another operator $L^{(z_2)}$, we obtain the sequence $v^{(2)}$ with characteristic polynomial $$(t-z_2)(t-(z_2+z_1))=t^2-(z_1+2z_2)t+z_2(z_1+z_2),$$ and zeros $\{z_2,z_2+z_1\}$, i.e., we obtain
$$v^{(2)}= L^{(z_2)}(r^{(1)})=(0,1,z_1+2z_2,\ldots). $$
For every $\alpha,\beta$ in a field $\mathbb F$ containing $R$, setting $z_2=\alpha$,  $z_1+z_2=\beta$, it is possible to \emph{L--construct} any impulsequence of order $2$, with the $2$--impulse $(0,1)$ as initial conditions. Repeating this process we can \emph{L--construct} any impulsequence of order $r$. 
The explained process is naturally invertible using the operator $\sigma$. So it is possible to \emph{L--deconstruct} every impulsequence into the startsequence. 
\begin{example}
The \emph{L--deconstruction} of the Fibonacci sequence \seqnum{A000045}. The characteristic polynomial of $F=(0,1,1,2,3,5,8,\ldots)$ is  $t^2-t-1$, with zeros $$\left\{\cfrac{1+\sqrt{5}}{2},\cfrac{1-\sqrt{5}}{2}\right\}.$$ Appliying $L^{\left(-\frac{1+\sqrt{5}}{2}\right)}$ we obtain
$$ L^{\left(-\frac{1+\sqrt{5}}{2}\right)}(F)=(0,1,-\sqrt{5},5,\ldots)=r^{(1)}, $$
with characteristic polynomial whose zeros are translated by $-\frac{1+\sqrt{5}}{2}$, i.e., $\{0,-\sqrt{5}\}$. Using the operator $\sigma$ we eliminate $0$ from the last sequence and from the set of zeros of its characteristic polynomial. The new sequence is
$$ \sigma(r^{(1)})=(1,-\sqrt{5},5,\ldots) $$
and the characteristic polynomial has the only zero $\{-\sqrt{5}\}$. In order to eliminate $\{-\sqrt{5}\}$ and to obtain the startsequence, we have to use the operator $L^{\left(\sqrt{5}\right)}$, as folows:
$$ L^{\left(\sqrt{5}\right)}(\sigma (r^{(1)}))=u. $$
Inverting this process we can \emph{L--construct} the Fibonacci sequence. Our purpose is to transform the characteristic polynomial $t$ of $u=r^{(0)}=(1,0,0,0,\ldots)$ into the polynomial $t^2-t-1$.
$$ L^{\left(-\sqrt{5}\right)}(r^{(0)})=(1,-\sqrt{5},5,\ldots) $$
has characteristic polynomial $t+\sqrt{5}$ with zero $\{-\sqrt{5}\}$ and
$$ \rho \left(L^{\left(-\sqrt{5}\right)}(r^{(0)})\right)=(0,1,-\sqrt{5},5,\ldots)=r^{(1)}, $$ 
has characteristic polynomial $t(t+\sqrt{5})$ with zeros $\{0,-\sqrt{5}\}$. In order to obtain the polynomial $t^2-t-1$, with zeros $\left\{\cfrac{1+\sqrt{5}}{2},\cfrac{1-\sqrt{5}}{2}\right\}$, finally we have to use $L^{\left(\frac{1+\sqrt{5}}{2}\right)}$ and we obtain
$$ L^{\left(\frac{1+\sqrt{5}}{2}\right)}(r^{(1)})=F. $$
\end{example}
It is possible to find an explicit expression of $v^{(k)}=L^{(z_k)}(r^{(k-1)})$ as shown in the following
\begin{proposition}
After $k$ steps of L-construction with parameters $z_1, z_2,\ldots,z_k \in \mathbb F$, $\mathbb F$ a field, we
obtain the recurrent sequence $v^{(k)}=L^{(z_k)}(r^{(k-1)})$ with characteristic polynomial 
\begin{equation}\label{charpoly}
(t - z_k )(t - z_k  - z_{k-1} ) \cdots (t - z_k  - z_{k-1}  -  \cdots  - z_1 )
\end{equation} 
and we have 
\begin{equation}\label{vkn}
(v^{(k)})_n=\sum\limits_{h_{k - 1}  = k - 1}^n {\sum\limits_{h_{k - 2}  = k - 2}^{h_{k - 1}  - 1}{ \cdots \sum\limits_{h_1  = 1}^{h_2  - 1}\binom{n}{h_{k - 1}}\binom{h_{k - 1}  - 1}{h_{k - 2} } \cdots \binom{h_2  - 1}{h_1 }  z_k^{n - h_{k - 1} } z_{k - 1}^{h_{k - 1}  - h_{k - 2}  - 1}  \cdots z_1^{h_1  - 1}}} 
\end{equation}
\end{proposition} 
\begin{proof}
Using an inductive argument, as previously explained, we have $v^{(1)}=L^{(z_1)}(r^{(0)})$, when $k=1$, with characteristic polynomial $f(t)=t-z_1$. If we suppose $v^{(k-1)}=L^{(z_{k-1})}(r^{(k-2)})$ with characteristic polynomial
$$(t - z_{k-1} )(t - z_{k-1}- z_{k-2} ) \cdots (t - z_{k-1}  - z_{k-2}  -  \cdots  - z_1 ),$$
we have $\rho(v^{(k-1)})=r^{(k-1)}$ with characteristic polynomial
$$t(t - z_{k-1} )(t - z_{k-1}- z_{k-2} ) \cdots (t - z_{k-1}  - z_{k-2}  -  \cdots  - z_1 )$$
and finally $v^{(k)}=L^{(z_k)}(r^{(k-1)})$ with characteristic polynomial given by (\ref{charpoly}).
Now let us consider $v^{(k)}$. From Definition \ref{binomial} and Definition \ref{sigma-rho}, observing that the first $k-1$ terms of $v^{(k)}$ are zeros
$$(v^{(k)})_n=\sum\limits_{h_{k - 1}  =k-1}^n {\binom{n}{h_{k - 1} }} z_k^{n - h_{k - 1} } (\rho(v^{(k - 1)}) )_{h_{k - 1} }=
\sum\limits_{h_{k - 1}  =k-1}^n {\binom{n}{h_{k - 1} }} z_k^{n - h_{k - 1} } (v^{(k - 1)} )_{h_{k - 1}-1 }. 
$$
Iterating this process of substitution, for $1\leq j\leq k-1$, we have
$$
(v^{(k - j)} )_{h_{k - j} }  = \sum\limits_{h_{k - j - 1}  = k - j - 1}^{h_{k - j}  - 1} {\binom{h_{k - j}  - 1}{h_{k - j - 1} }} z_{k - j}^{h_{k - j}  - h_{k - j - 1}  - 1} (v^{(k - j - 1)} )_{h_{k - j}  - 1} 
$$
we find easily (\ref{vkn}). 
\end{proof}
\begin{rema}
As an immediate consequence of the previous proposition we observe that for any $a \in\mathcal{W}(R)$ of order 2, with characteristic polynomial $f(t)$, we have  
$$ a_n=\sum_{i=1}^{n}\binom{n}{i}\alpha^{n-i}(\beta-\alpha)^{i-1}=\sum_{i=1}^{n}\binom{n}{i}\alpha^{n-i}(-\sqrt{\Delta})^{i-1}, \ \ \forall n \geq 1, $$
where $\alpha, \beta$ are the zeros of $f(t)$ and $\Delta$ its discriminant.
In fact, the sequence $a$ corresponds to a particular sequence $v^{(2)}$ constructed from $u$ setting $z_2=\alpha, z_1+z_2=\beta$. We also have incidentally retrieved the classical Binet formula
$$ \sum_{i=1}^{n}\binom{n}{i}\alpha^{n-i}(\beta-\alpha)^{i-1}=\cfrac{1}{\beta-\alpha}\sum_{i=0}^{n}\binom{n}{i}\alpha^{n-i}(\beta-\alpha)^{i}-\cfrac{\alpha^n}{\beta-\alpha}=\cfrac{\beta^n-\alpha^n}{\beta-\alpha}\quad.$$
\end{rema}
The processes of \emph{I--construction} and \emph{I--deconstruction} are similar to  \emph{L--construction} and \emph{L--deconstruction}. Starting from $u=r^{(0)}$, and using the operator $I^{(z_1)}$, we obtain the sequence with characteristic polynomial $t-z_1$, which is the same result obtained using $L^{(z_1)}$. Applying $\rho$ we have the sequence $r^{(1)}$, which we can write explicitly as follows:
$$ r^{(1)}=(0,1,z_1,z_1^2,z_1^3,\ldots), $$
and using the operator $I^{(z_2)}$, the resulting sequence
$$ I^{(z_2)}(r^{(1)})=(0,1,z_1,z_1^2+z_2,\ldots) $$
has characteristic polynomial $t(t-z_1)-z_2=t^2-z_1t-z_2$. Repeating or inverting this process we can \emph{I--construct} or \emph{I--deconstruct} any impulsequence of order r. So we have all the necessary tools to transform any impulsequence $a$ into any other impulsequence by means of a convenient composition of operators $\sigma, \rho, I^{(z)}, L^{(z)}$!
\section{Applications to integer sequences}
\subsection{Binomial anti--mean transform}

In this section we study the action of the operator
$L^{(-\frac{h}{2})}$, which we call \emph{Binomial anti--mean
transform} in analogy to the definition of the Binomial mean transform
\cite{Spivey}.

Let $W=W(s_0,s_1,h,k)$ be a linear recurrent sequence of degree 2 over $R$. Assume that its characteristic polynomial is $t^2-ht+k=(t-\alpha)(t-\beta)$ and the initial conditions are $s_0,s_1\in R$. We know that $L^{(y)}(W)$ has characteristic polynomial
$$ t^2-(h+2y)t+y^2+hy+k=(t-(\alpha+y))(t-(\beta+y)). $$
If we consider $y=-\cfrac{h}{2}$ , the binomial anti--mean transform $$L^{(-\frac{h}{2})}(W)=C$$ recurs with polynomial $t^2-\cfrac{\Delta}{4}$ , where $\Delta=h^2-4k$. For the sequence $C$  we have
$$ \begin{cases} C_0=s_0 \cr C_1=\cfrac{\delta}{2} \cr C_n=\cfrac{\Delta^{\lfloor \frac{n}{2}\rfloor}}{2^n}\delta^{n \bmod 2}s_0^{1-n \bmod 2}, \ \ \forall n \geq 2 \end{cases}$$
where $\delta=2s_1-s_0h$. Using Definition \ref{binomial}, we obtain 
$$ \sum_{i=0}^n\binom{n}{i}\left(-\cfrac{h}{2} \right)^{n-i}W_i=C_n.$$
If $s_0=0$, then $C_{2n}=0$ $\forall n$. This is the case of Fibonacci numbers $F=W(0,1,1,-1)$
$$ \sum_{i=0}^{2n}\binom{2n}{i}\left(-\cfrac{1}{2}\right)^{2n-i}F_i=0. $$

\subsection{r--bonacci numbers}
The operator Invert  acts on the startsequence giving
$$I(r^{(0)})=(1,1,1,1,\ldots),$$
with characteristic polynomial $t-1$. Using the operator $\rho$ we obtain
$$r^{(1)}=\rho(I(r^{(0)}))=(0,1,1,1,\ldots),$$ 
with characteristic polynomial $t^2-t$. Now, if we use Invert again, it is easy to observe that the result is $F=F^{(2)}$, the Fibonacci sequence \seqnum{A000045}
$$ I(r^{(1)})=F^{(2)}. $$
The Tribonacci sequence \seqnum{A000073}, $F^{(3)}= (0,0,1,1,2,4,7,13,\ldots)$, has characteristic polynomial $t^3-t^2-t-1$ and initial conditions $(0,0,1)$. So applying the operators $\rho$ and $I$ to the Fibonacci sequence we obtain the Tribonacci sequence 
$$ I(\rho(F^{(2)}))=F^{(3)}. $$
 From the Tribonacci sequence it is possible to obtain the Tetranacci sequence \seqnum{A000078} $F^{(4)}=(0,0,0,1,1,2,4,8,\ldots)$, as follows:
$$I(\rho(F^{(3)}))=F^{(4)}. $$
In general, let us start from the $r$--bonacci sequence $F^{(r)}$, whose recurrence polynomial is $t^r-t^{r-1}-t^{r-2}-\cdots-1$ and initial conditions are the \emph{r--impulse}. If we compose the operators $\rho$ and $I$, by Definition \ref{sigma-rho} and by (\ref{If}), it is simple to observe that the new sequence has characteristic polynomial $t^{r+1}-t^{r}-t^{r-1}-\cdots-1$ and initial conditions the $(r+1)$--\emph{impulse}, i.e., we obtain the $(r+1)$--bonacci sequence $F^{(r+1)}=I(\rho(F^{(r)}))$. Moreover, by Remark \ref{comb}, we can see that the $(r+1)$--bonacci numbers are the complete ordinary Bell polynomials of the $r$--bonacci numbers
$$ F^{(r+1)}_n=B_{n+1}(0,F^{(r)}_0,F^{(r)}_1,\ldots,F^{(r)}_n)\quad. $$
Finally, by Remark \ref{invertrec} we can use (\ref{invertreceq}), with $x=1$, and we find another relation between $r$--bonacci numbers
$$ F_{n+1}^{(r)}=F_n^{(r-1)}+\sum_{i=1}^{n-1}F_{i-1}^{(r)}F_{n-i-1}^{(r-1)}\quad. $$

\subsection{Polynomial sequences}
If we consider $f:\mathbb N\rightarrow \mathbb C$ and $d\in \mathbb N$ we have the following equivalent conditions \cite{Stanley}
\begin{itemize}
\item \begin{equation}
\sum\limits_{n \ge 0} {f(n)t^n }  = \frac{{P(t)}}{{(1 - t)^{d + 1} }},
\end{equation}
where $P(t)$ is a polynomial over $\mathbb C$ with $\deg(P)\leq d$
\item $\forall m\geq0$
\begin{equation}\label{recdelta}
\Delta ^{d + 1} f(m) = \sum\limits_{i = 0}^{d + 1} {( - 1)^{d + 1 - i} } \left( {\begin{array}{*{20}c}
   {d + 1}  \\
   i  \\
\end{array}} \right)f(m + i) = 0,
\end{equation}
\item $f(n)$ is a polynomial function of $n$, $\deg(f(n))\leq d$ and $\deg(f(n))=d\Leftrightarrow P(1)\neq0$. 
\end{itemize}
It is well known that it is possible to expand $f(n)$ in terms of the basis $n^i$, $0\leq i\leq d$, but an alternative way is to use the basis $\binom{n}{i}$, $ 0\leq i\leq d$.
In this case 
\begin{equation}\label{binfn}
f(n) = \sum\limits_{i = 0}^d {(\Delta ^i f(0)} )\left( {\begin{array}{*{20}c}
   n  \\
   i  \\
\end{array}} \right).
\end{equation}
Now we see another point of view to obtain (\ref{binfn}), based on the $L^{(y)}$ operator.
Let us consider a polynomial $f(t)$ over $\mathbb C$ such that $\deg(f)\leq d$. From (\ref{recdelta}), where we set $m=0$, we conclude that the characteristic polynomials of $$(f(n))=(f(0), f(1),f(2),\ldots)$$ and $$(\Delta ^n f(0))=(\Delta ^0 f(0),\Delta ^1 f(0)\Delta ^2 f(0),\ldots) $$ are, respectively, $(t-1)^{d+1}$ and $t^{d+1}$. Using (\ref{Lf}) it is clear that $$L^{( - 1)} ((t - 1)^{d + 1} ) = t^{d + 1}.$$ Consequently, we have the following \emph{one-click deconstruction}
$$L^{( - 1)} ((f(n)) ) = (\Delta ^n f(0)), $$
which implies
$$L^{(1)} ((\Delta ^n f(0))) = L((\Delta ^n f(0)) ) = (f(n)) $$
and directly gives (\ref{binfn}).

\subsection{Pyramidal numbers}

Let us consider the sequences $$(0,1,q-1,2q-3,3q-5,\ldots,1+(q-2)n,\ldots)\quad.$$ For all $q\geq2$, these sequences have characteristic polynomial $(t-1)^2$. Their partial sums are the \emph{polygonal numbers}. For example, when $q=3$, we have the triangular numbers $(0,1,3,6,10,15,\ldots)$ \seqnum{A000217}, when $q=4$ the square numbers $(0,1,4,9,16,\ldots)$ \seqnum{A000290}, and so on. The polygonal numbers have characteristic polynomial $(t-1)^3$. The $n$th term of such sequences is given by
\begin{equation}\label{pn}P^{(2)}_q(n)=
\cfrac{1}{2}(q-2)n^2+\cfrac{1}{2}(4-q)n .\end{equation}
If we calculate the partial sums of polygonal numbers, we  obtain the \emph{pyramidal numbers}. Their characteristic polynomial is $(t-1)^4$, and the $n$th term of such sequences is given by $P^{(3)}_q(n)$, a third degree polynomial. Proceeding with these partial sums, we can define pyramidal numbers for higher dimensions. The \emph{figurate numbers}, introduced in Remark \ref{fignum}, are a particular case of these numbers. These sequences recur with a polynomial $(t-1)^{d+1}$ and the terms of these sequences can be evaluated by $P^{(d)}_q(n)$, a polynomial of degree $d$. So we can apply the \emph{one--click deconstruction}.

For example, let us consider polygonal numbers $P^{(2)}_q(n)$. As we have observed, the polygonal numbers recur all with polynomial $(t-1)^3$, and equality (\ref{pn}) holds.
If we think of (\ref{pn}) as a function $f(n)$ in $n$, then we can apply the \emph{one--click deconstruction} to the polygonal numbers
$$ L^{(-1)}((P^{(2)}_q(n)))=(0,1,q-2,0,0,0,\ldots)\quad. $$
Indeed, in this case $\Delta^0f(0)=0$, $\Delta f(0)=1$, $\Delta^2f(0)=q-2$ \quad.
Consequently
\begin{equation} \label{Pn2} L^{(1)}((0,1,q-2,0,0,0,\ldots))=(P_q^{(2)}(n))\quad.\end{equation}
In the same way, starting from $1$, we obtain
\begin{equation} \label{Pn3} L^{(1)}((1,q-1,q-2,0,0,0,\ldots))=(P_q^{(2)}(n)), \quad n \geq 1 \quad . \end{equation}
The OEIS \cite{Sloane} shows special cases of (\ref{Pn2}) and (\ref{Pn3}), without proof. 
See for $q=4$ Sloane's \seqnum{A000566} (Paul Barry, 2003), and for $q=4,5,6,8,9,10$ Sloane's \seqnum{A000290},  \seqnum{A000326}, \seqnum{A000384}, \seqnum{A000567}, \seqnum{A001106}, \seqnum{A001107} (Gary W. Adamson, 2007-2008).

\bigskip
\hrule
\bigskip
\noindent 2000 {\it Mathematics Subject Classification}: Primary 11B37;
Secondary 11B39.

\noindent \emph{Keywords:}  Bell polynomials, binomial operator,
Fibonacci numbers, impulse sequences, invert operator, pyramidal
numbers, recurrent sequences.

\bigskip
\hrule
\bigskip

\noindent (Concerned with sequences 
\seqnum{A000045},
\seqnum{A000073},
\seqnum{A000078},
\seqnum{A000217},
\seqnum{A000290},
\seqnum{A000292},
\seqnum{A000326},
\seqnum{A000384},
\seqnum{A000566},
\seqnum{A000567},
\seqnum{A001106}, and
\seqnum{A001107}.)

\bigskip
\hrule
\bigskip

\vspace*{+.1in}
\noindent
Received September 14 2009;
revised versions received  February 15 2010; March 9 2010; June 29
2010; July 10 2010.
Published in {\it Journal of Integer Sequences},
July 16 2010.

\bigskip
\hrule
\bigskip

\noindent
Return to
\htmladdnormallink{Journal of Integer Sequences home page}{http://www.cs.uwaterloo.ca/journals/JIS/}.
\vskip .1in

\end{document}